\documentclass[11 pt]{amsart}
\usepackage{tikz}
\usepackage{tikz-cd}
\usepackage{amscd,amsfonts,amssymb,amsmath}
\usepackage{hyperref}
\usepackage{epsfig}
\newtheorem{theorem}{Theorem}[section]
\newtheorem{corollary}[theorem]{Corollary}

\newtheorem{proposition}[theorem]{Proposition}
\theoremstyle{definition}

\newtheorem{definition}[theorem]{Definition}

\numberwithin{equation}{subsection}

\newtheorem*{ack}{Acknowledgement}

\newcommand{\Aut}{\operatorname{Aut}}

\newcommand{\Conj}{\operatorname{Conj}}

\newcommand{\Inn}{\operatorname{Inn}}
\newcommand{\Der}{\operatorname{Der}}

\newcommand{\A}{\operatorname{A}}

\newcommand{\Hom}{\operatorname{Hom}}

\newcommand{\id}{\operatorname{id}}
\newcommand{\As}{\operatorname{As}}
\newcommand{\X}{\operatorname{X}}

\begin{document}
\title{Derivations of quandles}
\author{Neha Nanda}
\author{Mahender Singh}
\author{Manpreet Singh}

\address{Department of Mathematical Sciences, Indian Institute of Science Education and Research (IISER) Mohali, Sector 81,  S. A. S. Nagar, P. O. Manauli, Punjab 140306, India.}
\email{nehananda@iisermohali.ac.in}
\email{mahender@iisermohali.ac.in}
\email{manpreetsingh@iisermohali.ac.in}


\maketitle

\begin{abstract}
The aim of this paper is to propose a theory of derivations for quandles. Given a quandle $A$ admitting an action by a quandle $Q$, derivations from $Q$ to $A$ are introduced as twisted analogues of quandle homomorphisms. It is shown that for each quandle $Q$ there exists a unique $Q$-quandle $\mathcal{A}_Q$ (the derived quandle of $Q$) such that derivations from $Q$ to any $Q$-quandle $A$ are in bijective correspondence with $Q$-quandle homomorphisms from $\mathcal{A}_Q$ to $A$. Further, it is proved that the set of all derivations to an abelian $Q$-quandle $A$ has the structure of an abelian quandle, and inherits many other properties from $A$. In the end, the ideas are extended to the setting of virtual quandles.
\end{abstract}

\section{Introduction}\label{sec0}
A quandle is an algebraic system with a binary operation that satisfies certain axioms encoding the three Reidemeister moves of planar diagrams of links in the 3-space. Although special types of quandles had already appeared in the literature in different guises, the subject captured real attention after fundamental works of Joyce \cite{Joyce} and Matveev \cite{Matveev}. They showed that link quandles of non-split links are complete invariants up to orientation of the ambient space. This result has led to exploration and development of various aspects of the theory of quandles. The reader is encouraged to refer to \cite{Carter, Kamada, Nelson} for more on the historical development of the subject.
\par

The isomorphism problem for quandles is as hard as the classification problem for links. Thus, it is natural to develop newer invariants of quandles themselves. This makes study of algebraic aspects of quandles, which are not necessarily knot quandles, important. For example, a (co)homology theory for quandles and racks has been developed in \cite{Carter2}, which has led to stronger invariants of links. In fact, a recent work \cite{Szymik2019} shows that quandle cohomology is a Quillen cohomology which is the cohomology group of a functor from the category of models to that of complexes. Automorphisms of quandles, which reveal a lot about their internal structures have been investigated in much detail in a series of papers \cite{BDS, BarTimSin, Elhamdadi2012, Hou}.
\par

The aim of this paper is to propose a general theory of derivations for quandles. At this point, our motivation is purely algebraic and inspired by \cite{Crowell}, but the ideas might have applications to knot theory. The key point is the notion of action of a quandle on another quandle by automorphisms. The quandle action is then used to define derivations of quandles, which can be seen as twisted analogues of quandle homomorphisms, where the twisting is by quandle actions. For each quandle homomorphism we construct a universal quandle called the derived quandle. This derived quandle is then used to show finiteness of the set of derivations from a finitely generated quandle $Q$ to a finite quandle admitting a $Q$-action.
\par

The paper is organized as follows. Section \ref{prelim} collects some basic definitions from the theory of quandles. In Section \ref{UniversalPropertyOfDerivedQuandles}, using the notion of quandle actions, we introduce derivations of quandles (Definition \ref{derivation-definition}). We define the derived quandle of a quandle homomorphism (Definition \ref{derived-quandle-definition}) and prove its existence and uniqueness (Theorem \ref{derived-quandle-existence}). The universal property of the derived quandle is then used to prove that the total number of derivations from a finitely generated quandle $Q$ to a finite $Q$-quandle is always finite (Corollary \ref{number-derivations-finite}). In Section \ref{PropertiesOfDerivations}, we explore properties of derivations, and prove that the set of all derivations to an abelian quandle is an abelian quandle (Theorem \ref{der-set-abelian-quandle}), which is a generalisation of \cite[Theorem 4.1]{Crans-Nelson}. We also establish a general isomorphism theorem for the abelian quandle of derivations (Theorem \ref{MorphismBetweenSetOfDerivationsOverDifferentQuandles}). The paper conclude with Section \ref{section-virtual}, where the ideas are extended to the setting of virtual quandles (Theorem \ref{virtual-derived-existence-uniqueness}).
\medskip

\section{Preliminaries}\label{prelim}
A {\it quandle} is a non-empty set $Q$ with a binary operation $(q_1,q_2) \mapsto q_1 * q_2$ satisfying the following axioms:
\begin{enumerate}
\item[(Q1)] $q*q=q$ for all $q \in Q$;
\item[(Q2)] For any $q_1, q_2 \in Q$ there exists a unique $q \in Q$ such that $q_2=q*q_1$;
\item[(Q3)] $(q_1*q_2)*q_3=(q_1*q_3) * (q_2*q_3)$ for all $q_1,q_2,q_3 \in Q$.
\end{enumerate}
\par
Besides knot quandles associated to knots, many interesting examples of quandles come from groups. For example, if $G$ is a group, then the set $G$ equipped with the binary operation $$a*b= b a b^{-1}$$ gives a quandle structure on $G$, called the {\it conjugation quandle}, and denoted by $\Conj(G)$.
\par
A map $f : Q \to Q'$ is {\it quandle homomorphism} if $f(q_1 \ast q_2)= f(q_1) \ast f(q_2)$ for all $q_1 , q_2 \in Q$. Since $G \mapsto \Conj(G)$ is a functor from the category of groups to the category of quandles, any group homomorphism $f : G \to H$ induces a quandle homomorphism $f : \Conj(G) \to \Conj(H)$. Notice that, the quandle axioms are equivalent to saying that for each $q \in Q$, the map $S_q: Q \to Q$ given by $$S_q(p)=p*q$$ is an automorphism of the quandle $Q$ fixing $q$. Such an automorphism is called an {\it inner automorphism} of $Q$.
\par
For a given quandle $Q$, the group $\As(Q)$ generated by the set $\{ e_q ~|~ q \in Q \}$ modulo the relations $$e_{q_1 \ast q_2}=e_{q_2} e_{q_1} e_{q_2}^{-1}$$ is called the \textit{associated group} of the quandle $Q$. Note that there is a natural quandle homomorphism $$e : Q \to \Conj(\As(Q))$$ sending $q$ to $e_q$. Moreover, associated group $\As(Q)$ satisfies the property that given any quandle homomorphism $\varphi : Q \to \Conj(G)$, where $G$ is any group, there exists a unique group homomorphism $\bar{\varphi} : \As(Q) \to G$ such that the following diagram commutes 
\begin{equation}
\medskip
\label{UniversalPropertyOfAssociatedGroupOfQuandles}
\begin{tikzcd}
                                           &  & \Conj(\As(Q)) \arrow[dd, "\bar{\varphi}", dashed] \\
Q \arrow[rru, "e"] \arrow[rrd, "\varphi"'] &  &                                             \\
                                           &  & \Conj(G).                                  
\end{tikzcd}
\end{equation}

We shall write $\bar{\varphi}$ as $\varphi$ for convenience.
\medskip

\section{Derived Quandles}\label{UniversalPropertyOfDerivedQuandles}
Let $Q$ and $A$ be two quandles. We say that $A$ is a $Q$-quandle if there is an action of the associated group $\As(Q)$ on $A$ by quandle automorphisms. Equivalently, by the universal property of associated groups, $A$ is a $Q$-quandle if there is a quandle homomorphism $$\varphi : Q \to \Conj(\Aut(A)).$$
 A subquandle $B$ of $A$ is called a $Q$-subquandle if it is invariant under the action of $\As(Q)$. A quandle homomorphism $f: A_1 \rightarrow A_2$, where $A_1$ and $A_2$ are two $Q$-quandles, is called a $Q$-homomorphism if $$f(g \cdot a ) = g \cdot f(a)$$ for all $a \in A_1$ and $g \in \As(Q)$. The set of all such maps is denoted by $\Hom_Q(A_1,A_2)$.
\par

Since $Q \mapsto \As(Q)$ is a functor from the category of quandles to the category of groups, a quandle homomorphism $\phi: Q \rightarrow Q'$ induces a group homomorphism $\phi: \As(Q) \rightarrow \As(Q')$ given by $$\phi(e_q)=e_{\phi(q)}.$$ It is clear that every $Q'$-quandle $X$ is also a $Q$-quandle via $\phi$. Some examples of quandle actions are:
\begin{itemize}
\item There is a natural action of $\As(Q)$ on $Q$ by inner automorphisms. More precisely, the map $\varphi: \As(Q) \to \Aut(Q)$ given by $e_q \mapsto S_q$ is a group homomorphism.

\item If $X$ is any set, then any map $\varphi : \As(Q) \to \Sigma_X$, where $\Sigma_X$ is the symmetric group on the set $X$, defined by $e_q \mapsto \alpha$ for a fixed $\alpha \in \Sigma_X$, is a group homomorphism.
\end{itemize}
\medskip

According to \cite[Definition 2.2]{Rubinsztein}, an action of a topological quandle $Q$ on a topological space $X$ is a continuous map $Q \times X \to X$ given by $$(q, x) \mapsto q \cdot x,$$ and satisfying
\begin{equation}\label{rubinsztein-action}
(q_1 \ast q_2) \cdot (q_2 \cdot x) =  q_2 \cdot (q_1 \cdot x).
\end{equation}
It is easy to see that our definition of quandle action satisfies \eqref{rubinsztein-action}. However, the converse is not true. Let $Q$ be any quandle and $X= \{ 1, 2, 3\}$ with quandle matrix
$$
\setcounter{MaxMatrixCols}{20}
\begin{bmatrix}
     1&  3&  1 \\
   2&  2&  2 \\
   3&  1&  3 \\
\end{bmatrix}.
$$
Let $\sigma=(1,2,3)$ be a permutation of $X$. Notice that $\sigma$ is not a quandle homomorphism of $X$. The map $Q \times X \to X$ given by $(q ,x) \mapsto \sigma(x)$ for $x \in X, q \in Q$ clearly satisfies \eqref{rubinsztein-action}. But, for each $q \in Q$, the induced map $X \to X$ given by $x \mapsto \sigma(x)$, is not a quandle homomorphism.

\begin{definition}\label{derivation-definition}
Let $A$ be a $Q'$-quandle and $\phi: Q \rightarrow Q'$ a quandle homomorphism. A map $\partial: Q \rightarrow A$ is called a \textit{derivation} if 
$$
\partial(q_1 \ast q_2)= \partial(q_1) \ast \big(\phi(e_{q_1}) \cdot \partial(q_2) \big)
$$
for all $q_1, q_2 \in Q$.
\end{definition} 

Following are some examples of derivations:
\begin{itemize}
\item If the action of $Q'$ on $A$ is trivial, then every quandle homomorphism from $Q$ to $A$ is a derivation.
\item If $A$ is a trivial quandle, then for any action of $Q'$ on $A$,  every quandle homomorphism is a derivation.
\item Let $T_x \in \Aut(A)$ be such that it fix the point $x \in A$. Consider the action $Q \to \Conj\big(\Aut(A)\big)$ given by $q \mapsto T_x$ for all $q \in Q$. Then the constant map $Q \to A$ given as $q \mapsto x$, $q \in Q$, is clearly a derivation.
\end{itemize}

\begin{definition}\label{derived-quandle-definition}
 A \textit{derived quandle} of a quandle homomorphism $\phi: Q \rightarrow Q'$ is defined as a $Q'$-quandle $\mathcal{A}_\phi$ together with a derivation $\partial: Q \rightarrow \mathcal{A}_\phi$ such that for any $Q'$-quandle $A$ and a derivation $\partial': Q \rightarrow A$, there exists a unique $Q'$-homomorphism $\lambda : \mathcal{A}_\phi \to A$ such that $\lambda \circ \partial = \partial'$, that is, the following diagram commutes
\begin{equation}\label{DiagramOfUniversalPropertyOfDerivedQuandles}
\begin{tikzcd}
Q \arrow[rr, "\partial"] \arrow[rrdd, "\partial'"'] &  & \mathcal{A}_\phi \arrow[dd, "\lambda", dashed] \\
                                                    &  &                                      \\
                                                    &  & A.                                  
\end{tikzcd}
\end{equation}
\end{definition}
For our convenience, we denote the derived quandle of the quandle homomorphism $\phi$ by $(\mathcal{A}_\phi, \partial)$, a model for which we now construct explicitly.

\begin{theorem}\label{derived-quandle-existence}
The derived quandle of a quandle homomorphism exists and is unique.
\end{theorem}

\begin{proof}
Take $X = \As(Q') \times Q$ as a set and define $\mathcal{A}_\phi$ to be the quandle with presentation $$ \langle X ~|~ R~  \rangle,$$
where $R$ is the set of following relations $$(g, q_1 \ast q_2) = ( g, q_1 ) \ast \big(g \phi(e_{q_1}), q_2 \big) ,$$
for all $q_1, q_2 \in Q$ and $g \in \As(Q')$. We note that $\mathcal{A}_\phi$ is a $Q'$-quandle since there is a natural action of $\As(Q')$ on $\mathcal{A}_\phi$ which is defined on generators as $$g' \cdot (g, q) = (g' g,q),$$ where $(g, q) \in X$ and $g' \in \As(Q')$. We next define the map $\partial: Q \rightarrow \mathcal{A}_\phi$ by $$\partial (q)=(1,q).$$
For $q_1, q_2 \in Q$, we have
\begin{align*}
\partial (q_1 \ast q_2) &= (1, q_1 \ast q_2)\\
&=(1, q_1) \ast (\phi(e_{q_1}), q_2)\\
&=(1,q_1) \ast \big(\phi(e_{q_1}) \cdot(1,q_2) \big)\\
&= \partial(q_1) \ast \big(\phi(e_{q_1}) \cdot \partial(q_2) \big),
\end{align*}
and hence $\partial$ is a derivation.
\par 
Let us now consider any $Q'$-quandle $A$ and a derivation $\partial' : Q \to A$. We define a map $\lambda : \mathcal{A}_\phi \to A$ on generators by setting
$$\lambda \big((g,q) \big)= g \cdot \partial'(q).$$
The map $\lambda$ is well-defined since
\begin{align*}
\lambda \big((g, q_1 \ast q_2) \big) &= g \cdot \partial'(q_1 \ast q_2) \\
&= g \cdot \big(\partial'( q_1) \ast \big(\phi(e_{q_1}) \cdot \partial'(q_2) \big) \big)\\
&= g \cdot \partial'(q_1) \ast \big( (g \phi(e_{q_1})) \cdot \partial'(q_2) \big)\\
&=\lambda\big((g,q_1)\big) \ast \lambda \big( \big( g\phi(e_{q_1}) \big),q_2 \big).
\end{align*}
By definition, it follows that $\lambda$ is a $Q'$-homomorphism such that $\lambda \circ \partial = \partial'$.

The map $\lambda$ is unique. For, if there exists another $Q'$-homomorphism $\lambda' : \A_\phi \to A$ such that $\lambda' \circ \partial = \partial'$, then we have
 $$ \lambda' \big(\partial (q) \big) = \partial'(q) $$
for all $q \in Q$. That is,
$$ \lambda' \big((1,q) \big)= \partial'(q). $$
Since $\lambda'$ is $Q'$-homomorphism, we have
\begin{align*}
\lambda' \big((g,q) \big) &=  g \cdot \lambda' \big((1,q) \big)\\
&= g \cdot \partial'(q)\\
&= \lambda \big((g,q) \big).
\end{align*}
Thus, the derived quandle $(\mathcal{A}_\phi, \partial)$ of the quandle homomorphism $\phi : Q \to Q'$ always exists, and is clearly unique upto isomorphism.
\end{proof}

If $Q=Q'$ and $\phi=\id$, we denote $\mathcal{A}_\phi$ by $\mathcal{A}_Q$.

\begin{definition}
Let $A$ be a $Q$-quandle and $S \subset A$.  We say that $A$ is \textit{$Q$-generated} by $S$ if $A$ is generated by the set $\{ g \cdot s ~|~ g \in \As(Q), ~ s \in S\}$ as a quandle. If such a finite set exists, then we say that $A$ is a \textit{finitely generated $Q$-quandle}.
\end{definition}

For example, let $Q$ be a quandle viewed as a $Q$-quandle under the natural action of $\Inn(Q)$. If $S$ is a set of representatives of orbits under this action, then $Q$ is $Q$-generated by the set $S$. Further, if $Q$ has finitely many orbits, then it is a finitely generated $Q$-quandle.

\begin{theorem}\label{MainTheoremOfDerivedQuandles}
Let $\phi : Q \to Q'$ be a quandle homomorphism. If $Q$ is generated by $S$, then $\mathcal{A}_\phi$ is $Q'$-generated by $\partial(S)$.
\end{theorem}
 
\begin{proof}
Let $A$ be a $Q'$-subquandle of $\mathcal{A}_\phi$ generated by the set $\partial(S)$. Since $S$ generates the quandle $Q$ and $A$ is invariant under the action of $\As(Q')$, it follows that $\partial(q) \in A$ for all $q \in Q$. Therefore, we can define a map $\partial' : Q \to A$ by $$\partial'(q) = \partial(q),$$ for all $q \in Q$, which is a derivation. By the universal property of derived quandles, there exists unique $Q'$-homomorphism $\lambda : \mathcal{A}_\phi \to A$ such that $\lambda \circ \partial = \partial'$. If $\iota : A \to \mathcal{A}_\phi$ is the natural inclusion, then we have the following diagram
$$
\begin{tikzcd}
                                                                                   &  & \mathcal{A}_\phi \arrow[d, "\lambda"] \\
Q \arrow[rru, "\partial"] \arrow[rrd, "\partial"', dashed] \arrow[rr, "\partial'"] &  & A \arrow[d, "\iota", hook]  \\
                                                                                   &  & \mathcal{A}_\phi.              
\end{tikzcd}
$$
Thus, for every $q \in Q$, have $\iota \circ \lambda \circ \partial(q) = \partial(q)$. By uniqueness, $\iota \circ \lambda = \id$, which implies that $\iota$ is surjective, and hence $A = \mathcal{A}_\phi$.
\end{proof}
 
 As a consequence, we have the following result.
 
 \begin{corollary}\label{MainCorollaryOfDerivedQuandles}
Let $\phi : Q \to Q'$ be a quandle homomorphism. If $Q$ is a finitely generated quandle, then $\mathcal{A}_\phi$ is a finitely generated $Q'$-quandle.
 \end{corollary}
 
Using the preceding corollary and the universal property of derived quandles, we obtain the following result.

\begin{corollary}\label{number-derivations-finite}
Let $Q$ be a finitely generated quandle, $\phi : Q \to Q'$ a quandle homomorphism and $A$ a $Q'$-quandle which is finite. Then the number of derivations from $Q$ to $A$ is finite.
\end{corollary}
\medskip

\section{Properties of derivations}\label{PropertiesOfDerivations}
Let $A$ be a $Q'$-quandle and $\phi: Q \rightarrow Q'$ a quandle homomorphism. We denote the set of all derivations from $Q$ to $A$ by $\Der_\phi(Q, A)$. In this section, we investigate properties of $\Der_\phi(Q, A)$ that it inherits from given quandles, particularly $A$. 

Recall that a quandle $A$ is said to be \textit{abelian} if
$$ (x \ast y) \ast ( z \ast w) = (x \ast z) \ast ( y \ast w)$$
for all  $ x, y, z,w \in A$.
Abelian quandles are also called {\it medial quandles} in the literature. See \cite{JPSZ} for a recent study of these quandles.

\begin{theorem}\label{der-set-abelian-quandle}
Let $\phi : Q \to Q'$ be a quandle homomorphism and $A$ an abelian $Q'$-quandle. If the set $\Der_\phi(Q, A)$ is non-empty, then it has the structure of an abelian quandle with respect to the binary operation given by 
\begin{equation}\label{point-wise1}
(f \ast g)(q) = f(q) \ast g(q),
\end{equation}
where $f, g \in \Der_\phi(Q, A)$ and $q \in Q$.
\end{theorem}

\begin{proof}
We note that the set of all maps from $Q$ to $A$ forms an abelian quandle under the operation defined in \eqref{point-wise1}. It now suffices to show that $\Der_\phi(Q,A)$ is closed under $\ast$ and $\ast^{-1}$. Let $f, g \in \Der_\phi(Q,A)$ and $q_1, q_2 \in Q$. Then we have 
\begin{align*}
(f \ast g)(q_1 \ast q_2) &= f(q_1 \ast q_2) \ast g(q_1 \ast q_2)\\
&= \big( f(q_1) \ast \big( \phi(e_{q_1}) \cdot f(q_2)\big) \big) \ast \big( g(q_1) \ast \big( \phi(e_{q_1}) \cdot g(q_2) \big) \big)\\
&= \big( f(q_1) \ast g(q_1) \big) \ast \big( \phi(e_{q_1}) \cdot f(q_2) \ast \phi(e_{q_1}) \cdot g(q_2)\big), ~\text{since }A \text{ is abelian}\\
&= (f \ast g)(q_1) \ast \big(\phi(e_{q_1}) \cdot (f \ast g)(q_2)\big).
\end{align*}
Similarly, we can show that
$$(f \ast^{-1} g)(q_1 \ast q_2) = (f \ast^{-1}  g)(q_1) \ast \big(\phi(e_{q_1}) \cdot (f \ast^{-1}  g)(q_2)\big),$$
which is desired.
\end{proof}

If the action of $Q$ on $A$ (via $\phi$) is trivial, then $\Der_\phi(Q,A) = \Hom(Q, A)$ and we recover the following result of \cite[Theorem 4.1]{Crans-Nelson}.

\begin{theorem}
Let $Q$ be a quandle and $A$ an abelian quandle. Then the set $\Hom(Q, A)$ has the structure of an abelian quandle with respect to the binary operation as given in \eqref{point-wise1}.
\end{theorem}

\begin{proposition}
Let $\phi : Q \to Q'$ be a quandle homomorphism and $A$ an abelian $Q'$-quandle. Then the set $\Hom_{Q'}(\mathcal{A}_\phi, A)$ forms an abelian quandle under the binary operation defined by $$(\lambda_1 \ast \lambda_2 ) \big((g,q) \big) = \lambda_1 \big((g,q) \big) \ast \lambda_2 \big((g,q) \big),$$
where $(g,q) \in \mathcal{A}_\phi$ and $\lambda_1, \lambda_2 \in \Hom_{Q'}(\mathcal{A}_\phi, A)$.
\end{proposition}

\begin{proof}
It suffices to prove that the set $\Hom_{Q'}(\mathcal{A}_\phi, A)$ is closed under $\ast$ and $\ast^{-1}$. 
For $\lambda_1, \lambda_2 \in \Hom_{Q'}(\mathcal{A}_\phi, A)$, $q_1, q_2 \in Q$ and $g \in \As(Q')$, we have 
\begin{align*}
\big(\lambda_1 \ast \lambda_2\big) \big( (g, q_1 \ast q_2) \big)&= \lambda_1 \big( (g, q_1 \ast q_2) \big) \ast \lambda_2 \big( (g, q_1 \ast q_2) \big)\\
&= \lambda_1\big((g, q_1) \ast (g\phi(e_{q_1}), q_2)\big) \ast  \lambda_2\big((g, q_1) \ast (g\phi(e_{q_1}), q_2)\big)\\
&= \big(\lambda_1 \big( (g, q_1) \big) \ast \lambda_1 \big( (g\phi(e_{q_1}), q_2) \big) \big) \ast   \big( \lambda_2 \big( (g, q_1) \big) \ast \lambda_2 \big( (g\phi(e_{q_1}), q_2)\big) \big)\\
&= \big( \lambda_1 \big( (g, q_1) \big) \ast  \lambda_2 \big((g, q_1)\big) \big) \ast \big( \lambda_1 \big( (g\phi(e_{q_1}), q_2) \big) \ast \lambda_2 \big( (g\phi(e_{q_1}), q_2)\big) \big),  ~A \text{ is abelian }\\ 
&=(\lambda_1 \ast \lambda_2) \big((g,q_1) \big) \ast (\lambda_1 \ast \lambda_2) \big( (g\phi(e_{q_1}), q_2) \big).
\end{align*}
Similarly, we can show that
$$\big(\lambda_1 \ast^{-1} \lambda_2\big) \big( (g, q_1 \ast q_2) \big) = (\lambda_1 \ast^{-1} \lambda_2) \big((g,q_1) \big) \ast (\lambda_1 \ast^{-1} \lambda_2) \big( (g\phi(e_{q_1}), q_2) \big).$$
\end{proof}

By the universal property of derived quandles, there is a bijection between the sets $\Der_\phi(Q,A)$ and $\Hom_{Q'}(\mathcal{A}_\phi, A)$. Below we show that this bijection is, in fact, an  isomorphism of quandles provided $A$ is abelian. 

\begin{proposition}\label{IsomorphismOfMorphismsDerivations}
Let $\phi : Q \to Q'$ be a quandle homomorphism and $A$ an abelian $Q'$-quandle. Then $$\Hom_{Q'}(\mathcal{A}_\phi, A)\cong \Der_\phi(Q, A).$$
\end{proposition}

\begin{proof}
Let $(A_\phi, \partial)$ be the derived quandle of the quandle homomorphism $\phi$ and $$\tilde \partial : \Hom_{Q'}(\mathcal{A}_\phi, A) \to \Der_\phi(Q,A)$$ be defined as
$$\tilde \partial (\lambda) =  \lambda \circ \partial,$$
where $\lambda \in \mathcal{Q}$. By existence and uniqueness of $(A_\phi, \partial)$, it follows that $\tilde{\partial}$ is a bijection. Further, for $\lambda_1, \lambda_2 \in \Hom_{Q'}(\mathcal{A}_\phi, A)$ and $q \in Q$, we have 
\begin{align*}
\big(\tilde \partial (\lambda_1 \ast \lambda_2) \big)(q) &= (\lambda_1 \ast \lambda_2) \big((1, q) \big)\\
&= \lambda_1 \big( (1,q) \big) \ast \lambda_2 \big((1,q) \big)\\
&= \tilde \partial (\lambda_1) \ast \tilde \partial (\lambda_2)(q),
\end{align*}
which is desired.
\end{proof}

Recall that a quandle $Q$ is said to be \textit{commutative} if 
$ p \ast q = q \ast p $ for all $p, q \in Q$. A quandle is called \textit{involutary} if $S_q^2 = \id_Q$ for all $q \in Q$. For example, all Takasaki quandles are involutary. Motivated by the work of Loos \cite{Loos1, Loos2} on Riemannian symmetric spaces, Ishihara and Tamaru \cite{Ishihara} defined a quandle $X$ to be \textit{flat} if the group $$\big\langle S_p \circ S_q \mid p,q \in Q\big\rangle$$ is abelian. It was shown in \cite{Singh} that Takasaki quandles of 2-divisible groups are flat. If $A$ is an abelian quandle, then we have the following result whose proof is immediate.

\begin{proposition}\label{flat-inheritance}
The following statements hold for quandles $\Der_\phi(Q,A)$ and $\Hom_{Q'}(\mathcal{A}_\phi, A)$:
\begin{enumerate}
\item If $A$ is commutative, then $\Der_\phi(Q, A)$ and $\Hom_{Q'}(\mathcal{A}_\phi, A)$ are commutative.
\item If $A$ is involutary, then $\Der_\phi(Q, A)$ and $\Hom_{Q'}(\mathcal{A}_\phi, A)$ are involutary.
\item If $A$ is flat, then $\Der_\phi(Q, A)$ and $\Hom_{Q'}(\mathcal{A}_\phi, A)$ are flat.
\end{enumerate}
\end{proposition}

\begin{theorem}\label{MorphismBetweenSetOfDerivationsOverDifferentQuandles}
Let $\phi_1 : Q_1 \to X_1$, $\phi_2:Q_2 \to X_2$, $\sigma: Q_2 \to Q_1$ and $\psi: X_1 \to X_2$ be quandle homomorphisms such that $\psi \circ \phi_1 \circ \sigma = \phi_2$. Let $A_i$ be an abelian $\X _i$-quandle for $i=1,2$. If $\tau: A_1 \to A_2$ is a $X_1$-homomorphism, then there exist quandle homomorphisms  $$\Phi: \Der_{\phi_1}(Q_1, A_1) \to \Der_{\phi_2}(Q_2, A_2),$$
$$\tilde \Phi: \Hom_{X_1}(\mathcal{A}_{\phi_1}, A_1) \to \Hom_{X_2}(\mathcal{A}_{\phi_2}, A_2).$$
Further, if $\psi$, $\sigma$ and $\tau$ are isomorphisms, then so are $\Phi$ and $\tilde \Phi$.
\end{theorem}

\begin{proof}

For each $f \in \Der_{\phi_1}(Q_1, A_1)$, define 
$$\Phi: \Der_{\phi_1}(Q_1, A_1) \to \Der_{\phi_2}(Q_2, A_2),$$
by setting
$$\Phi(f)= \tau \circ f \circ \sigma.$$
We first check that the map $\Phi$ is well-defined. For any $q_1, q_2 \in Q_2$, we have
\begin{align*}
\Phi(f)(q_1 \ast q_2) &= \tau \circ f \circ \sigma (q_1 \ast q_2)\\
&= \tau \circ f \big(\sigma (q_1) \ast \sigma(q_2) \big)\\
&= \tau \Big( f \big(\sigma(q_1) \big) \ast \phi_1(e_{\sigma(q_1)}) \cdot f \big(\sigma(q_2) \big) \Big)\\
&= \tau  f  \sigma (q_1) \ast \Big(\psi \big(\phi_1(e_{\sigma(q_1)}) \big) \cdot \tau f \sigma(q_2) \Big)\\
&= \Phi(f)(q_1) \ast \Big(\phi_2(e_{q_1}) \cdot \Phi(f)(q_2) \Big),  ~ \text{since } \psi \circ \phi_1 \circ \sigma = \phi_2.
\end{align*}
It follows easily that $\Phi$ is a quandle homomorphism. Further, if $\sigma$, $\psi$ and $\tau$ are isomorphisms, then we can define 
$$\Psi: \Der_{\phi_2}(Q_2, A_2) \to \Der_{\phi_1}(Q_1, A_1)$$
by setting
$$\Psi(f)= \tau^{-1} \circ f \circ \sigma^{-1}$$
for $f \in \Der_{\phi_2}(Q_2, A_2)$. An easy check shows that $\Phi$ and $\Psi$ are inverses of each other.\\
Now we define 
$$\tilde \Phi: \Hom_{X_1}(\mathcal{A}_{\phi_1}, A_1) \to \Hom_{X_2}(\mathcal{A}_{\phi_2}, A_2)$$
by setting
$$\tilde \Phi = {\tilde \partial_2}^{-1} \circ \Phi \circ \tilde \partial_1,$$
where $\tilde \partial_1 : \Hom_{X_1}(\mathcal{A}_{\phi_1}, A_1) \to \Der_{\phi_1}(Q_1, A_1)~~\textrm{and}~~\tilde \partial_2 : \Hom_{X_2}(\mathcal{A}_{\phi_2}, A_2) \to \Der_{\phi_2}(Q_2, A_2)$ are quandle isomorphisms as defined in Proposition \ref{IsomorphismOfMorphismsDerivations}. It is immediate that $\tilde \Phi$ is also a quandle homomorphism, and if $\sigma$, $\psi$ and $\tau$ are isomorphisms, then so is $\tilde{\Phi}$.
\end{proof}
\medskip

\begin{corollary}\label{main-cor1}
Let $Q$ be a quandle, $A_1,A_2$ two abelian $Q$-quandles and $\tau : A_1 \to A_2$ a $Q$-homomorphism. Then there exist quandle homomorphisms $$\Phi: \Der_{\id}(Q, A_1) \to \Der_{\id}(Q, A_2),$$ 
$$ \tilde \Phi: \Hom_Q(\mathcal{A}_{Q}, A_1) \to \Hom_Q(\mathcal{A}_{Q}, A_2).$$
Further, if $\tau$ is an isomorphism, then so are $\Phi$ and $\tilde \Phi$.
\end{corollary}
\medskip

The proof of the following result follows along similar lines as that of Theorem \ref{MorphismBetweenSetOfDerivationsOverDifferentQuandles}.

\begin{proposition}\label{main-cor2} 
Let $\sigma: Q_2 \to Q_1$ be a quandle homomorphism and $A$ an abelian $Q_1$-quandle. If $A$ is viewed as a $Q_2$-quandle via $\sigma$, then there exist  quandle homomorphisms $$\Phi : \Der_{\id}(Q_1, A) \to \Der_{\sigma}(Q_2, A),$$
$$ \tilde \Phi: \Hom_{Q_1}(\mathcal{A}_{Q_1}, A) \to \Hom_{Q_2}(\mathcal{A}_{Q_2}, A).$$
Moreover, if $\sigma$ is an isomorphism, then so are $\Phi$ and $\tilde \Phi$.
\end{proposition}
\medskip

\section{Virtual derived quandles}\label{section-virtual}

Recall from \cite{Manturov} that a \textit{virtual quandle}, denoted $(Q,\alpha)$, is a quandle $Q$ together with an automorphism $\alpha \in \Aut(Q)$. If $\alpha = \id_Q$, then we obtain the usual quandle. A map $f : (Q,\alpha) \to (Q',\alpha ')$ is called a \textit{virtual quandle homomorphism} if $f(q_1 \ast q_2)= f(q_1) \ast f(q_2)$, for all $q_1 , q_2 \in Q$, and $$\alpha ' \circ f=f \circ \alpha.$$
We denote the set of automorphisms of the virtual quandle $(Q,\alpha)$ by $\Aut(Q, \alpha)$. Henceforth, for convenience, we shall drop the word virtual. 

\begin{definition}
The {\it associated group} of a quandle $(Q,\alpha)$ is defined as the pair $(\As(Q), \alpha)$, where $\alpha$ is the induced automorphism of $\As(Q)$ given by $\alpha(e_q) \mapsto e_{\alpha(q)}$.
\end{definition}

\begin{proposition}\label{UniversalPropertyOfVirtualQuandle}
Let $(Q, \alpha)$ be a quandle, $G$ a group and $\beta \in \Aut(G)$. If $\varphi : (Q,{\alpha}) \to \big(\Conj(G), \beta\big)$ is a quandle homomorphism, then there exists a unique group homomorphism $\lambda: \As(Q) \to G$ such that the following diagrams commute

$$
\begin{tikzcd}
                                                              &  & {\big(\Conj(\As(Q)), \alpha\big)} \arrow[dd, "\lambda", dashed] &  & \As(Q) \arrow[dd, "\alpha"'] \arrow[rr, "\lambda"] &  & G \arrow[dd, "\beta"] \\
{(Q,\alpha)} \arrow[rru, "e"] \arrow[rrd, "\varphi"'] &  &                                                         &  &                                                    &  &                       \\
                                                              &  & {\big(\Conj(G),\beta\big)},                                      &  & \As(Q) \arrow[rr, "\lambda"']                      &  & G.                    
\end{tikzcd}
$$
\end{proposition}

\begin{proof}
Define $\lambda: \As(Q) \to G$ on generators of $\As(Q)$ by $\lambda(e_q) = \varphi(q)$ for all $q \in Q$. It is easy to check that $\lambda$ is a group homomorphism and $\lambda \circ e = \varphi$. Further, for any $e_q \in \As(Q)$, we have 
\begin{align*}
\beta \circ \lambda(e_q) &= \beta (\varphi(q))\\
&= \varphi \circ \alpha(q)\\
&= \lambda (e_{\alpha(q))}\\
&= \lambda \circ \alpha(e_q).
\end{align*}
Uniqueness of the map $\lambda$ is clear.
\end{proof}

We say that a quandle $(Q, \alpha)$ acts on a quandle $(A, \gamma)$ if there exists a group homomorphism $\varphi: \As(Q) \to \Aut(A,\gamma)$ such that the following diagram commutes
$$
\begin{tikzcd}
\As(Q) \arrow[dd, "\alpha"'] \arrow[rr, "\varphi"] &  & \Aut(A,\gamma) \arrow[dd, "S_{\gamma}"] \\
                                                   &  &                                 \\
\As(Q) \arrow[rr, "\varphi"']                      &  & \Aut(A,\gamma).                      
\end{tikzcd}
$$
Equivalently, by Proposition \ref{UniversalPropertyOfVirtualQuandle}, $(Q, \alpha)$ acts on $(A, \gamma)$ if there is a quandle homomorphism $$\varphi : (Q,\alpha) \to \big( \Conj(\Aut(A,\gamma)), S_{\gamma}\big),$$ where $S_\gamma$ is the inner automorphism of the quandle $\Conj(\Aut(A,\gamma))$.
A quandle $(A,\gamma)$ is said to be a $(Q,\alpha)$-quandle if there is an action of the quandle $(Q,\alpha)$ on $(A,\gamma)$. We note that a quandle homomorphism $\phi : (Q, \alpha) \to (Q', \alpha')$ induces a group homomorphism $\phi: \As(Q) \to \As(Q')$ such that
\begin{equation}\label{equation}
\alpha' \circ \phi = \phi \circ \alpha.
\end{equation}
It is easy to check that every $(Q',\alpha')$-quandle is also a $(Q,\alpha)$-quandle by commutativity of the following diagram
$$
\begin{tikzcd}
\As(Q) \arrow[r, "\phi"] \arrow[d, "\alpha"] & \As(Q') \arrow[r, "\varphi"] \arrow[d, "\alpha'"] & \Aut(A,\gamma) \arrow[d, "S_\gamma"] \\
\As(Q) \arrow[r, "\phi"']                    & \As(Q') \arrow[r, "\varphi"']                     & \Aut(A,\gamma).                      
\end{tikzcd}
$$
A quandle homomorphism $\phi : (Q_1, \alpha_1) \to (Q_2, \alpha_2)$, where $(Q_1, \alpha_1)$ and $(Q_2, \alpha_2)$ are $(Q, \alpha)$-quandles, is said to be $(Q, \alpha)$-homomorphism if $$\phi(g \cdot q) = g \cdot \phi(q),$$ for all $q \in Q_1$ and $g \in As(Q)$.

\begin{definition}
Let $(A,\gamma)$ be a $(Q',\alpha')$-quandle and $\phi : (Q, \alpha) \to (Q', \alpha')$ be a quandle homomorphism. A map $\partial: (Q,\alpha) \to (A,\gamma)$ is called a (\textit{virtual}) \textit{derivation} if for all $q_1, q_2$ in $Q$,
\begin{equation}\label{ConditionForDerivationOfVirtualQuandle}
\partial(q_1 \ast q_2)= \partial(q_1) \ast \big (\phi(e_{q_1})\cdot \partial(q_2) \big) ~\text{and}
\end{equation}
$$\gamma \circ \partial = \partial \circ \alpha.$$
\end{definition}
 
\begin{definition} A \textit{(virtual) derived quandle} of the quandle homomorphism $\phi : (Q, \alpha) \to (Q', \alpha')$ consists of a $(Q',\alpha')$-quandle $(\mathcal{A}_\phi, \Gamma)$ and a derivation $\partial: (Q,\alpha) \to (\mathcal{A}_\phi, \Gamma)$ such that for any $(Q',\alpha')$-quandle $(A, \gamma)$ and a derivation $\partial': (Q,\alpha) \to (A,\gamma)$, there exists a unique $(Q',\alpha')$-homomorphism $\lambda : (\mathcal{A}_\phi,\Gamma) \to (A,\gamma)$ making the following diagram commute

\begin{equation}\label{DiagramOfUniversalPropertyOfVirtualQuandles}
\begin{tikzcd}
{(Q,\alpha)} \arrow[rr, "\partial"] \arrow[rrdd, "\partial'"'] &  & {(\mathcal{A}_\phi,\Gamma)} \arrow[dd, "\lambda", dashed] \\
                                                               &  &                                                 \\
                                                               &  & {(A,\gamma).}                                  
\end{tikzcd}
\end{equation}
\end{definition}

For our convenience, we denote the derived quandle of the quandle homomorphism $\phi$ by $(\mathcal{A}_\phi, \Gamma, \partial)$.\\ 

\begin{theorem}\label{virtual-derived-existence-uniqueness}
The virtual derived quandle of a virtual quandle homomorphism exists and is unique.
\end{theorem}
\begin{proof}

Let $\mathcal{A}_\phi$ be as defined in Section \ref{UniversalPropertyOfDerivedQuandles}. Define a map $\Gamma : \mathcal{A}_\phi \to \mathcal{A}_\phi$ on generators by
$$\Gamma \big((g,q) \big) = \big(\alpha'(g),\alpha(q) \big),$$ 
where $q \in Q$ and $g \in \As(Q')$. We first check that $\Gamma$ is indeed an automorphism of $\mathcal{A}_\phi$ for which it suffices to show that $$\Gamma \big((g, q \ast q') \big) = \Gamma \big((g, q) \big) \ast \Gamma \big((g \phi(e_q),q') \big)$$ for any $q \in Q$, $q' \in Q'$ and $g \in \As(Q')$.
Considering the left hand side, we get 
\begin{align*}
\Gamma \big((g, q \ast q') \big) &= \big(\alpha'(g), \alpha(q \ast q') \big)\\
&= \big(\alpha'(g), \alpha(q) \ast \alpha(q') \big)\\
&= \big(\alpha'(g), \alpha(q) \big) \ast \big(\alpha'(g)\phi(e_{\alpha(q)}), \alpha(q') \big)\\
&= \Gamma \big((g,q) \big) \ast \big(\alpha'(g)\phi(\alpha(e_q)),\alpha(q') \big)\\
&=\Gamma \big((g,q) \big) \ast \big(\alpha'(g)\alpha'(\phi(e_q)), \alpha(q') \big), ~\textrm{since}~\phi~\textrm{is a quandle homomorphism}\\ 
&=\Gamma \big((g, q) \big) \ast \Gamma \big((g \phi(e_q),q') \big).
\end{align*}
\par
Next, we define a group homomorphism $\varphi : \As(Q') \to \Aut(\mathcal{A}_\phi,\Gamma)$ by setting
$$\varphi(g') \big((g,q) \big) = (g'g, q),$$
for $(g,q) \in \As(Q') \times Q$ and $g' \in \As(Q')$, for which
\begin{align*}
(\Gamma \circ \varphi(g')) \big((g,q) \big) &= \Gamma \big( (g'g,q) \big)\\
&= \big(\alpha'(g'g),\alpha(q)\big)\\
&= \big(\alpha'(g')\alpha'(g),\alpha(q)\big)\\
&= \varphi \big(\alpha'(g') \big) \big((\alpha'(g),\alpha(q)) \big)\\
&= (\varphi\big(\alpha'(g')\big) \circ \Gamma)\big((g,q) \big).
\end{align*}
Therefore, $(\mathcal{A}_\phi, \Gamma)$ is a $(Q', \alpha')$-quandle. It is easy to check that the map $\partial : (Q,\alpha) \to (\mathcal{A}_\phi, \Gamma)$ defined by $$\partial(q) = (1,q)$$ satisfies \eqref{ConditionForDerivationOfVirtualQuandle} and $\Gamma \circ \partial = \partial \circ \alpha$, and hence is a derivation.
\par 
Let us now consider any $(Q', \alpha')$-quandle $(A, \gamma)$ and a derivation $\partial' : (Q, \alpha) \to (A, \gamma)$. We recall the map $\lambda : \mathcal{A}_\phi \to A$ (from Section \ref{UniversalPropertyOfDerivedQuandles}) given by
$$\lambda \big((g,q) \big)= g \cdot \partial'(q).$$
Then, for any $(g,q) \in \mathcal{A}_{\phi}$, we have
\begin{align*}
(\gamma \circ \lambda)\big( (g,q) \big) &= \gamma \big( g \cdot \partial'(q) \big)\\
&=\alpha'(g) \cdot \gamma \big(\partial'(q) \big),~\textrm{since}~A ~\textrm{is a}~(Q',\alpha')\textrm{-quandle}\\
&=\alpha'(g) \cdot \partial' \big(\alpha (q) \big), ~\textrm{since}~\partial'~\textrm{is a derivation}  \\
&=\lambda \big((\alpha' (g), \alpha(q))\big)\\
&=\lambda \circ \Gamma \big((g,q) \big),
\end{align*}
which shows that $\lambda : (\mathcal{A}_\phi, \Gamma) \to (A, \gamma)$ is a $(Q', \alpha')$-homomorphism.  It follows that $\lambda$ is the unique $(Q',\alpha')$-homomorphism for which the diagram \eqref{DiagramOfUniversalPropertyOfVirtualQuandles} commutes. Hence, $(\mathcal{A}_\phi, \Gamma, \partial)$ always exists, and is unique upto isomorphism.
\end{proof}

\begin{theorem}
 Let $(Q, \alpha)$ be a finitely generated quandle and $\phi : (Q, \alpha) \to (Q', \alpha')$ a quandle homomorphism. Then $(\mathcal{A}_\phi, \Gamma)$ is a finitely generated $(Q', \alpha')$-quandle.
\end{theorem}

\begin{proof}
The proof follows along  similar lines as that of Theorem \ref{MainTheoremOfDerivedQuandles} and Corollary \ref{MainCorollaryOfDerivedQuandles}.
\end{proof}

\begin{corollary}
Let $(Q, \alpha)$ be a finitely generated quandle, $\phi : (Q, \alpha) \to (Q', \alpha')$ a quandle homomorphism and $(A, \gamma)$ a finite quandle which is also a $(Q',\alpha')$-quandle. Then there exists only finitely many $(Q', \alpha')$-homomorphisms from $(\mathcal{A}_\phi, \Gamma)$ to $(A, \gamma)$. Further, the number of derivations from $(Q, \alpha)$ to $(A, \gamma)$ is finite.
\end{corollary}
\medskip

\begin{ack}
Neha Nanda and Manpreet Singh would like to thank IISER Mohali for the PhD Research Fellowship. Mahender Singh would like to acknowledge support from SERB MATRICS Grant MTR/2017/000018.
\end{ack}


\medskip


\begin{thebibliography}{HD}
\bibitem{BDS} V. G. Bardakov, P. Dey and M. Singh, \textit{Automorphism groups of quandles arising from groups},  Monatsh. Math. 184 (2017), 519--530.
\bibitem{BarTimSin} V. G. Bardakov, T. Nasybullov and M. Singh, \textit{Automorphism groups of quandles and related groups},  Monatsh. Math. 189 (2019), no. 1, 1--21.
\bibitem{Carter} J. Carter, \textit{A survey of quandle ideas}, Introductory lectures on knot theory, 22--53, Ser. Knots Everything, 46, World Sci. Publ., Hackensack, NJ, 2012.
\bibitem{Carter2} J. Scott Carter, Daniel Jelsovsky, Seiichi Kamada, Laurel Langford and Masahico Saito, \textit{Quandle cohomology and state-sum invariants of knotted curves and surfaces}, Trans. Amer. Math. Soc. 355 (2003), no. 10, 3947--3989.
\bibitem{Crans-Nelson} A. S. Crans and S. Nelson, \textit{Hom quandles}, J. Knot Theory Ramifications 23 (2014),1450010, 18 pp.
\bibitem{Crowell} R. H.  Crowell, \textit{The derived module of a homomorphism}, Adv. Math. 6 (1971), 210--238.
\bibitem{Elhamdadi2012} M. Elhamdadi, J. Macquarrie and R.  Restrepo, \textit{Automorphism groups of quandles}, J. Algebra Appl. 11 (2012), 1250008, 9 pp.
\bibitem{Hou} Xiang-dong Hou,  \textit{Automorphism groups of Alexander quandles}, J. Algebra 344 (2011), 373--385.
\bibitem{Ishihara} Y. Ishihara and H. Tamaru, \textit{Flat connected finite quandles}, Proc. Amer. Math. Soc. 144 (2016), 4959--4971.
\bibitem{JPSZ}  P. Jedli\v{c}ka, A. Pilitowska, D. Stanovsk\'{y} and A. Zamojska-Dzienio, \textit{The structure of medial quandles}, J. Algebra 443 (2015), 300--334. 
\bibitem{Joyce} D. Joyce, \textit{A classifying invariant of knots, the knot quandle}, J. Pure Appl. Algebra 23 (1982), 37--65.
\bibitem{Kamada} S. Kamada, \textit{Knot invariants derived from quandles and racks}, Invariants of knots and 3-manifolds (Kyoto, 2001), 103--117, Geom. Topol. Monogr., 4, Geom. Topol. Publ., Coventry, 2002.
\bibitem{Loos1} O. Loos, \textit{Reflexion spaces and homogeneous symmetric spaces}, Bull. Amer. Math. Soc. 73 (1967), 250--253.  
\bibitem{Loos2}  O. Loos, \textit{Spiegelungsr\"{a}ume und homogene symmetrische R\"{a}ume}, Math. Z. 99 (1967), 141--170. 
\bibitem{Manturov} V. O. Manturov and D. P. Ilyutko, \textit{Virtual knots. The state of the art}, Translated from the 2010 Russian original. With a preface by Louis H. Kauffman. Series on Knots and Everything, 51. World Scientific Publishing Co. Pte. Ltd., Hackensack, NJ, 2013. xxvi+521 pp.
\bibitem{Matveev} S. Matveev, \textit{Distributive groupoids in knot theory}, (Russian) Mat. Sb. (N.S.) 119 (161), 78--88, 160 (1982).
\bibitem{Nelson} S. Nelson, \textit{The combinatorial revolution in knot theory}, Notices Amer. Math. Soc. 58 (2011), 1553--1561. 
\bibitem{Rubinsztein} R. Rubinsztein, \textit{Topological quandles and invariants of links},  J. Knot Theory Ramifications  16 (2007), 789--808.
\bibitem{Singh} M. Singh, \textit{Classification of flat connected quandles},  J. Knot Theory Ramifications 25 (2016), 1650071, 8 pp.
\bibitem{Szymik2019} M. Szymik, \textit{Quandle homology is Quillen homology},  Trans. Amer. Math. Soc. 371 (2019), no. 8, 5823--5839.
\end{thebibliography}
\end{document}